\documentclass[a4paper,leqno]{amsart}

\usepackage{amssymb,amsthm,amsmath,amsrefs}

\usepackage{ifthen}

\newcounter{mylisti} \newcounter{mylistii}
\newcounter{nest}
\newcommand{\defaultlabel}{}

\newenvironment{mylist}[1]{%
  \addtocounter{nest}{1}
  \ifthenelse{\value{nest}=1}{%
    \renewcommand{\defaultlabel}{(\roman{mylisti})\hfill}}{%
    \renewcommand{\defaultlabel}{(\alph{mylistii})\hfill}}
  \begin{list}{\defaultlabel}{%
      \ifthenelse{\value{nest}=1}{\usecounter{mylisti}}{%
        \usecounter{mylistii}}
      
      \addtolength{\itemsep}{0.5ex}
      \settowidth{\labelwidth}{#1}
      \setlength{\leftmargin}{\labelwidth}
      \addtolength{\leftmargin}{\labelsep}}}{\addtocounter{nest}{-1}
\end{list}}

\usepackage{accents}
\newcommand{\thickbar}[1]{\accentset{\rule{.4em}{.6pt}}{#1}}
\newcommand{\widethickbar}[1]{\accentset{\rule{.6em}{.6pt}}{#1}}

\newcommand{\Widethickbar}[1]{\accentset{\rule{.8em}{.6pt}}{#1}}

\newcommand{\lb}{\ensuremath{\thickbar{l}}}
\newcommand{\mb}{\ensuremath{\Widethickbar{m}}}
\newcommand{\nb}{\ensuremath{\thickbar{n}}}
\newcommand{\pb}{\ensuremath{\thickbar{p}}}

\newcommand{\bn}{\ensuremath{\mathbb N}}

\newcommand{\br}{\ensuremath{\mathbb R}}

\newcommand{\cE}{\ensuremath{\mathcal E}}
\newcommand{\cF}{\ensuremath{\mathcal F}}

\newcommand{\cP}{\ensuremath{\mathcal P}}

\newcommand{\cU}{\ensuremath{\mathcal U}}
\newcommand{\cV}{\ensuremath{\mathcal V}}

\usepackage{mathrsfs}

\newcommand{\cUb}{\ensuremath{\widethickbar{\mathcal U}}}

\newcommand{\vP}{\ensuremath{\boldsymbol{\mathrm{P}}}}

\newcommand{\vS}{\ensuremath{\boldsymbol{\mathrm{S}}}}

\newcommand{\abs}[1]{\lvert #1\rvert}

\newcommand{\Bigabs}[1]{\Big\lvert #1\Big\rvert}
\newcommand{\biggabs}[1]{\bigg\lvert #1\bigg\rvert}

\DeclareMathOperator*{\wlim}{w\mathrm{-}\lim}

\newcommand{\sizeleq}[2]{{[#1]}^{\leq #2}}
\newcommand{\sizeeq}[2]{{[#1]}^{#2}}

\newcommand{\infin}[1]{{[#1]}^{\omega}}

\newcommand{\itws}{\text{\textcircled z}}

\DeclareMathOperator{\lip}{\ensuremath{\mathrm{Lip}}}

\newcommand{\norm}[1]{\lVert #1\rVert}
\newcommand{\bignorm}[1]{\big\lVert #1\big\rVert}
\newcommand{\Bignorm}[1]{\Big\lVert #1\Big\rVert}

\newcommand{\restrict}{\ensuremath{\!\!\restriction}}

\newcommand{\tsum}{\ts\sum}

\newcommand{\co}{\mathrm{c}_0}

\newcommand{\V}{\forall \,}

\newcommand{\vare}{\varepsilon}

\renewcommand{\geq}{\geqslant}
\renewcommand{\leq}{\leqslant}

\newcommand{\ds}{\displaystyle}

\newcommand{\ts}{\textstyle}

\newcommand{\ie}{\textit{i.e.,}\ }

\newcommand{\etc}{\textit{etc.}}

\newtheorem{thm}{Theorem}
\newtheorem{problem}{Problem}

\newtheorem{lem}[thm]{Lemma}
\newtheorem{prop}[thm]{Proposition}
\newtheorem{cor}[thm]{Corollary}

\theoremstyle{definition}

\newtheorem*{defn}{Definition}

\theoremstyle{remark}

\newtheorem*{rem}{Remark}
\newtheorem*{rems}{Remarks}

\makeatletter
\def\author@andify{%
  \nxandlist {\unskip ,\penalty-1 \space\ignorespaces}%
    {\unskip {} \@@and~}%
    {\unskip \penalty-2 \space \@@and~}%
}
\makeatother

\makeatletter
\renewcommand{\andify}{%
  \nxandlist{\unskip, }{\unskip{} \@@and~}{\unskip{} \@@and~}}
\makeatother

\makeatletter
\renewcommand{\PrintNames@a}[4]{%
    \PrintSeries{\name}
        {#1}
        {}{ and \set@othername}
        {,}{ \set@othername}
        {}{ and \set@othername}
        {#2}{#4}{#3}%
}
\makeatother

\author{F.~Baudier}
\address{Department of Mathematics, Texas A\&M University, College
  Station, TX 77843, USA}
\email{florent@tamu.edu}

\author{Th.~Schlumprecht}
\address{Department of Mathematics, Texas A\&M University, College
  Station, TX 77843, USA and Faculty of Electrical Engineering, Czech
  Technical University in Prague, Zikova 4, 166 27, Prague, Czech
  Republic}
\email{schlump@tamu.edu}

\author{A.~Zs\'ak}
\address{Peterhouse, Cambridge CB2 1RD and Department of Pure
  Mathematics and Mathematical Statistics, Centre for Mathematical
  Sciences, University of Cambridge, Wilberforce Road, Cambridge CB3
  0WB, United Kingdom}
\email{a.zsak@dpmms.cam.ac.uk}

\thanks{F.~Baudier was supported by the National Science Foundation
  under Grant Number DMS-2055604. Th.~Schlumprecht was supported by
  the National Science Foundation under Grant Number
  DMS-2054443. A.~Zs\'ak was supported by the Workshop in Analysis and
  Probability at Texas A\&M University in 2023.}
	


\title{On stability of metric spaces and Kalton's property $Q$}

\begin{document}

\begin{abstract}
  The first named author introduced the notion of upper stability for
  metric spaces in~\cite{baudier:22} as a relaxation of stability. The
  motivation was a search for a new invariant to distinguish the class
  of reflexive Banach spaces from stable metric spaces in the coarse
  and uniform category. In this paper we show that property~$Q$ does
  in fact imply upper stability. We also provide a direct proof of the
  fact that reflexive spaces are upper stable by relating the latter
  notion to the asymptotic structure of Banach spaces.
\end{abstract}

\maketitle

\section{Introduction}

In~\cite{kalton:07}, Kalton proved two landmark results about the
coarse and uniform geometry of infinite-dimensional Banach spaces. The
first result was a nonlinear embedding result that states that every
stable metric space admits a coarse and uniform embedding into a
reflexive Banach space.

While uniform embeddings have long been considered in nonlinear Banach
space geometry (see~\cite{ben-lin:00} for an extensive
account), the notion of coarse embedding was introduced more recently
by Gromov (under a different terminology) in~\cite{gromov:93} as a
notion of metric faithfulness that turned out to be of paramount
importance in geometric group theory and noncommutative geometry. We
refer to the monograph~\cite{now-yu:12} for a further discussion of
coarse embeddings. Kalton's embedding result was quantitatively
refined in~\cite{bau-lan:15}, where it was observed that the
embeddings possess even better metric faithfulness properties: they
can be arbitrarily close to isometric embeddings.

Krivine and Maurey introduced the notion of stability for separable
Banach spaces~\cite{kri-mau:81} in order to extend a result of
Aldous~\cite{aldous:81} by showing that every stable Banach space
admits an isomorphic copy of $\ell_p$ for some
$p\in[1,\infty)$. Garling~\cite{garling:82} then studied this
isometric Banach space notion in the more general context of arbitrary
metric spaces.

\begin{defn}
  A metric space $(M,d_M)$ is \emph{stable }if for any pair $\cU,\cV$
  of free ultrafilters on $\bn$ and for any two bounded sequences
  $(x_n),(y_n)$ in $M$ we have
  \[
  \lim_{m,\cU}\lim_{n,\cV}d_M(x_m,y_n)=\lim_{n,\cV}\lim_{m,\cU}d_M(x_m,y_n)\ .
  \]
\end{defn}

A close connection between stability and reflexivity in the category
of metrizable topological spaces and continuous maps was known
since the original work of Krivine and Maurey, and the new insight of
Kalton was to establish another connection in the category of metric
spaces and Lipschitz maps. In the latter connection, certain spaces of
Lipschitz functions play a decisive role. Both connections rely on an
important factorization result for weakly compact operators via
reflexive spaces~\cite{dav-fig-joh-pel:74}. Kalton raised a few
natural problems.

\begin{problem}[\cite{kalton:07}*{Problem~6.1}]
  \label{problem:kalton1}
  Does every (separable) reflexive Banach space coarsely (or
  uniformly) embed into a stable metric space? Does every reflexive
  Banach space admit an embedding, which is simultaneously coarse and
  uniform, into a stable metric space?
\end{problem}

Kalton explicitly suspected that Problem~\ref{problem:kalton1} might
have a negative solution, but that Problem~\ref{problem:kalton2} below
could have a positive solution.

\begin{problem}[\cite{kalton:07}*{Problem~6.2}]
  \label{problem:kalton2}
  If $X$ is a separable reflexive Banach space, does $B_X$ embed
  uniformly into a stable metric space?
\end{problem}

We will now mainly focus on the coarse case of
Problem~\ref{problem:kalton1} (as the discussion of the uniform case
is similar).  Recall that a metric space $(M,d_M)$ \emph{coarsely
embeds }into a metric space $(N,d_N)$ if there exist a map
$f\colon M\to N$ and non-decreasing functions
$\omega,\rho\colon[0,\infty)\to[0,\infty)$ with
$\lim_{t\to\infty}\rho(t)=\infty$ such that for all $x,y\in M$ we have
\begin{equation}
  \label{eq:coarse-embedding}
  \rho\big(d_M(x,y)\big)\leq
  d_N\big(f(x),f(y)\big)\leq\omega\big(d_M(x,y)\big)\ .
\end{equation}

A time-tested method to provide a negative answer to
Problem~\ref{problem:kalton1} is to find a coarse invariant, \ie a
property that is preserved under coarse embeddings, that is possessed
by every stable metric space but fails for at least one separable
reflexive Banach space. In~\cite{kalton:07}, Kalton introduced several
variants of a concentration property for Lipschitz maps defined on
certain infinite graphs, namely the \emph{interlacing graphs}. Given
an infinite subset $L$ of $\bn$ and an integer $k\in \bn$, we consider
the set $\sizeeq{L}{k}$ of all subsets of $L$ of size $k$ as the
vertex set of a graph equipped with the interlacing adjacency
relation, where two vertices $\mb=\{m_1,\dots,m_k\}$ and
$\nb=\{n_1,\dots,n_k\}$ in $\sizeeq{L}{k}$ with $m_1<\dots<m_k$ and
$n_1<\dots<n_k$ are adjacent if and only if they \emph{interlace}, \ie
$\mb\neq\nb$ and either
$m_1\leq n_1 \leq m_2\leq n_2\leq\dots\leq m_k\leq n_k$ or
$n_1\leq m_1\leq n_2\leq m_2\leq\dots\leq n_k\leq m_k$. The graph
metric induced by the interlacing adjacency relation will be denoted
by $d_I$.
\begin{defn}
  A metric space $(M,d_M)$ is said to have Kalton's \emph{property~$Q$
  }if there is a constant $C\geq 0$ such that for all $k\in\bn$ and
  for all bounded functions $f\colon(\sizeeq{\bn}{k},d_I)\to(M,d_M)$,
  there exists $L\in\infin{\bn}$ such that
  \[
  \sup_{\mb,\nb\in\sizeeq{L}{k}} d_M\big(f(\mb),f(\nb)\big)\leq
  C\lip(f)\ .
  \]
  The infimum of all constants $C$ satisfying the above will be
  denoted $Q_M$ and called the \emph{property~$Q$ constant of $M$}.
\end{defn}
Here $\infin{\bn}$ denotes the set of all infinite subsets of $\bn$
and $\lip(f)$ is the Lipschitz constant of $f$. Kalton's original
definition of property~$Q$ for Banach spaces was given in terms of a
modulus~\cite{kalton:07}*{pages~403--404}, but eventually took the
form above in~\cite{kalton:13}*{page~1057} (see
also~\cite{kalton:12}*{page~1283}).

The study of concentration inequalities for Lipschitz maps defined on
infinite graphs with countable degree, which began with property~$Q$
in~\cite{kalton:07} for interlacing graphs and continued shortly
thereafter in~\cite{kal-ran:08} for Hamming graphs, was a genuine
breakthrough towards unlocking many mysteries regarding the nonlinear
geometry of Banach spaces in terms of their infinite-dimensional
and asymptotic properties
(see~\citelist{\cite{bau-kal-lan:10}\cite{kalton:11}\cite{bcdkrsz:17}\cite{lan-raj:18}\cite{bau-lan-sch:18}\cite{netillard:18}\cite{bra-lan-pet-pro:23}\cite{lan-pet-pro:20}\cite{bau-lan-mot-sch:21a}\cite{bau-lan-mot-sch:21b}\cite{bau-gar:21}\cite{fovelle:23}}).

Kalton observed~\cite{kalton:07}*{Proposition~3.1} that every stable
metric space has property~$Q$, and that if a Banach space $X$ admits a
coarse embedding into a metric space with property~$Q$, then $X$ has
property~$Q$. The interlacing graph metric can be seen as being
induced by the summing basis of $\co$, and thus every Banach space
containing an isomorphic copy of $\co$ contains bi-Lipschitz copies of
the sequence $\big(\sizeeq{\bn}{k},d_I\big)_{k=1}^\infty$ of
interlacing graphs with uniformly bounded distortion. It follows that
$\co$ does not embed coarsely into any metric space with
property~$Q$. The second remarkable result of Kalton
from~\cite{kalton:07} is that reflexive Banach spaces also have
property~$Q$. An important consequence is that $\co$ does not coarsely
embed into any reflexive Banach space, thereby solving a longstanding
open question in the coarse geometry of Banach spaces. Kalton also
showed that property~$Q$ is a uniform as well as coarse invariant for
Banach spaces: if a Banach space $X$, or its closed unit ball $B_X$,
admits a uniform embedding into a metric space with property~$Q$, then
$X$ has property~$Q$. It follows that $\co$, or indeed the closed unit
ball of $\co$, does not uniformly embed into any reflexive Banach space,
solving yet another old open problem in the nonlinear theory of Banach
spaces.

The fact that reflexive spaces have property~$Q$ rules out
property~$Q$ as a potential coarse invariant to resolve negatively
Problem~\ref{problem:kalton1}. In an attempt to attack Problem
\ref{problem:kalton1}, the first author introduced
in~\cite{baudier:22} another invariant called upper stability. This
bi-Lipschitz invariant is a natural relaxation of the notion of
stability in light of a characterization of stability due to
Raynaud~\cite{raynaud:83} (see Theorem~\ref{thm:raynaud}
below). Upper stability was shown to imply property~$Q$ and to be
preserved under coarse embeddings (for Banach space domains and metric
space targets akin to property~$Q$).

\begin{defn}
  \label{def:upper-stable}
  Let $(M,d_M)$ be a metric space and $C$ be a non-negative real
  number. We say that $M$ is \emph{$C$-upper stable }if for all
  integers $1\leq l<k$, bounded functions
  $f\colon\sizeeq{\bn}{l}\to M$ and $g\colon\sizeeq{\bn}{k-l}\to M$,
  infinite sets $L\subset\bn$ and permutations $\pi\in S_k$ preserving
  the order on $\{1,\dots,l\}$ and on $\{l+1,\dots,k\}$,
    \begin{multline*}
      \inf_{\nb\in\sizeeq{L}{k}}
      d_M\big(f(n_1,\dots,n_l),g(n_{l+1},\dots,n_k)\big) \leq\\
      C\cdot\sup_{\nb\in\sizeeq{L}{k}}
      d_M\big(f(n_{\pi(1)},\dots,n_{\pi(l)}),g(n_{\pi(l+1)},\dots,n_{\pi(k)})\big)\ .
    \end{multline*}
    We say that $M$ is \emph{upper stable }if it is $C$-upper stable
    for some $C$.
\end{defn}

There are good reasons to think that an approach to solving
Problem~\ref{problem:kalton1} via upper stability would work. Indeed,
if $X$ is a Banach space containing no copies of $\ell_p$ for
$1\leq p<\infty$, then $X$ cannot isomorphically embed into a stable
Banach space because of the result of Krivine and Maurey mentioned
above. If in addition $X$ is separable, then $X$ is not even
bi-Lipschitzly embeddable into a reflexive stable Banach space due to
a result of Heinrich and
Mankiewicz~\cite{hei-man:82}*{Theorem~3.5}. Thus, any such space $X$
(Tsirelson's space, Schlumprecht's space, Gowers-Maurey space, \etc)
could be a candidate for a reflexive space failing upper stability. In
this paper, we dash this hope in two different ways, and at the same
time, we further our understanding of the connection between
stability, reflexivity and property~$Q$.

In Section~\ref{sec:upper-stability-and-asymptotic-structure}, we
prove that every reflexive Banach space is upper stable. Since
upper stability implies property~$Q$ (see~\cite{baudier:22}), we thus
recover the fact, originally proved by Kalton in~\cite{kalton:07},
that reflexive spaces have property~$Q$. Kalton's original proof
relies on some clever ``calculus'' for certain limiting operators in
reflexive spaces. In order to prove the a~priori stronger property of
upper stability, we take a more conceptual approach as we relate
upper stability with the asymptotic structure of Banach spaces in the
sense of Maurey, Milman and Tomczak-Jaegermann~\cite{mau-mil-tom:95}.

First, however, we prove in
Section~\ref{sec:upper-stability-and-property-Q} the more general
result that property~$Q$ and upper stability are equivalent notions
(up to the value of the constants involved) in the category of metric
spaces.

\section{Upper stability and property~$Q$}
\label{sec:upper-stability-and-property-Q}

We begin by stating the following characterization of stability, which
combines work of Krivine and Maurey~\cite{kri-mau:81} and
Raynaud~\cite{raynaud:83} (see also~\cite{baudier:22}).

\begin{thm}
  \label{thm:raynaud}
  The following are equivalent for a metric space $(M,d_M)$.
  \begin{mylist}{(iii)}
  \item
    $M$ is stable.
  \item
    For all $1\leq l<k$, bounded functions $f\colon\sizeeq{\bn}{l}\to M$
    and $g\colon\sizeeq{\bn}{k-l}\to M$, free ultrafilters
    $\cU_1,\dots,\cU_k$ and permutations $\pi\in S_k$ preserving the
    order on $\{1,\dots,l\}$ and on $\{l+1,\dots,k\}$,
    \begin{multline*}
      \lim_{n_1,\cU_1}\dots\lim_{n_k,\cU_k}
      d_M\big(f(n_1,\dots,n_l),g(n_{l+1},\dots,n_k)\big) =\\
      \lim_{n_{\pi^{-1}(1)},\cU_{\pi^{-1}(1)}}\dots\lim_{n_{\pi^{-1}(k)},\cU_{\pi^{-1}(k)}}
      d_M\big(f(n_1,\dots,n_l),g(n_{l+1},\dots,n_k)\big)\ .      
    \end{multline*}
  \item
    For all $1\leq l<k$, bounded functions $f\colon\sizeeq{\bn}{l}\to M$
    and $g\colon\sizeeq{\bn}{k-l}\to M$, free ultrafilters $\cU$ and
    permutations $\pi\in S_k$ preserving the order on $\{1,\dots,l\}$
    and on $\{l+1,\dots,k\}$,
    \begin{multline*}  
      \lim_{n_1,\cU}\dots\lim_{n_k,\cU}
      d_M\big(f(n_1,\dots,n_l),g(n_{l+1},\dots,n_k)\big) =\\
      \lim_{n_1,\cU}\dots\lim_{n_k,\cU}
      d_M\big(f(n_{\pi(1)},\dots,n_{\pi(l)}),g(n_{\pi(l+1)},\dots,n_{\pi(k)})\big)\ .
    \end{multline*}
  \item
    For all $1\leq l<k$, bounded functions $f\colon\sizeeq{\bn}{l}\to M$
    and $g\colon\sizeeq{\bn}{k-l}\to M$, infinite sets $L\subset\bn$
    and permutations $\pi\in S_k$ preserving the order on
    $\{1,\dots,l\}$ and on $\{l+1,\dots,k\}$,
    \begin{multline*}
      \inf_{\nb\in\sizeeq{L}{k}}
      d_M\big(f(n_1,\dots,n_l),g(n_{l+1},\dots,n_k)\big) \leq\\
      \sup_{\nb\in\sizeeq{L}{k}}
      d_M\big(f(n_{\pi(1)},\dots,n_{\pi(l)}),g(n_{\pi(l+1)},\dots,n_{\pi(k)})\big)\ .
    \end{multline*}
  \end{mylist}
\end{thm}

In light of the above characterization, the definition of
upper stability given in the Introduction is a very natural relaxation
of stability. There is a certain asymmetry in part~(iv) above, and
hence in the definition of upper stability, in the sense that the
permutation is applied only to the right-hand side of the defining
inequality. We will now introduce an a priori stronger version of
upper stability by ``symmetrizing'' this inequality. We will also
introduce a different formulation in terms of cuts instead of
permutations.

Fix $k,l,m\in\bn$ with $k=l+m$. By a \emph{cut of $\{1,\dots,k\}$ of
size $l$ }we mean simply a subset $P$ of $\{1,\dots,k\}$ of size
$l$. If $P=\{p_1,\dots,p_l\}$ with $p_1<\dots<p_l$, then for
$n_1<\dots<n_k$ in $\bn$, we write $(n_i:\,i\in P)$ for
$(n_{p_1},\dots,n_{p_l})$. The proof of Raynaud's characterization of
stability shows the following.

\begin{prop}
  \label{prop:cut-stable-characterization}
  For a metric space $(M,d_M)$ and a constant $C\geq1$, the following
  are equivalent.
  \begin{mylist}{(ii)}
  \item
    For all $k,l,m\in\bn$ with $k=l+m$, free ultrafilters $\cU$,
    bounded functions $f\colon\sizeeq{\bn}{l}\to M$ and
    $g\colon\sizeeq{\bn}{m}\to M$ and cuts $P,Q$ of $\{1,\dots,k\}$ of
    size $l$, we have
    \begin{multline*}
      \lim_{n_1,\cU}\dots\lim_{n_k,\cU}
      d_M\big(f(n_i:\,i\in P),g(n_i:\,i\in P^c)\big)\leq\\
      C\cdot \lim_{n_1,\cU}\dots\lim_{n_k,\cU}
      d_M\big(f(n_i:\,i\in Q),g(n_i:\,i\in Q^c)\big)\ .
    \end{multline*}
  \item
    For all $k,l,m\in\bn$ with $k=l+m$, bounded functions
    $f\colon\sizeeq{\bn}{l}\to M$ and $g\colon\sizeeq{\bn}{m}\to M$,
    infinite sets $L\subset\bn$ and cuts $P,Q$ of $\{1,\dots,k\}$ of
    size $l$, we have
    \begin{multline*}
      \inf_{\nb\in\sizeeq{L}{k}}d_M\big(f(n_i:\,i\in P),g(n_i:\,i\in
      P^c)\big)\leq\\
      C\cdot\sup_{\nb\in\sizeeq{L}{k}}d_M\big(f(n_i:\,i\in Q),g(n_i:\,i\in
      Q^c)\big)\ .
    \end{multline*}
  \end{mylist}
\end{prop}

\begin{defn}
  We say a metric space $(M,d_M)$ is \emph{$C$-cut stable }if it
  satisfies one of the two equivalent conditions in
  Proposition~\ref{prop:cut-stable-characterization}. We say $M$ is
  \emph{cut stable }if it is $C$-cut stable for some $C\geq1$.
\end{defn}

\begin{rems}
  We recover the notion of upper stability by taking $P$ to be the
  trivial cut $\{1,2,\dots,l\}$ above. We will later see that every
  upper stable metric space is cut stable, although this is not
  entirely trivial. It is clear that upper stability is a Lipschitz
  invariant, and hence in particular an isomorphic invariant for
  Banach spaces. In contrast, stability is an isometric invariant. It
  is easy to find equivalent norms on the stable Banach space
  $\ell_1$, for instance, which is not stable.
\end{rems}

As mentioned in the Introduction, Kalton's original definition of
property~$Q$ was given in terms of a modulus, which we now recall.

\begin{defn}[Kalton~\cite{kalton:07}]
  Let $(M,d_M)$ be a metric space. Given $\vare>0$ and $\delta\geq 0$,
  we say that $M$ has the \emph{$Q(\vare,\delta)$-property }if for all
  $k\in\bn$ and for all functions
  $f\colon\big(\sizeeq{\bn}{k},d_I\big)\to M$ with
  $\lip(f)\leq\delta$, there exists $L\in\infin{\bn}$ such that
  \[
  d_M(f(\mb),f(\nb))<\vare\qquad\text{for all }\mb<\nb\text{ in
  }\sizeeq{L}{k}\ .
  \]
  For $\vare>0$ we define
  \[
  \Delta_M(\vare)=\sup\{\delta\geq 0:\,M \text{ has the
    $Q(\vare,\delta)$-property}\}\ ,
  \]
  and set $q_M$ to be the supremum of all constants $c\geq 0$ for
  which $\Delta_M(\vare)\geq c\vare$ for all $\vare>0$.
\end{defn}

\begin{rem}
  It is immediate from above that $M$ has property~$Q$ if and only if
  $q_M>0$ in which case $q_M^{-1}\leq Q_M\leq2q_M^{-1}$. Kalton proved
  in~\cite{kalton:07} that for a stable metric space $M$, $q_M\geq1$,
  and for a reflexive Banach space $X$, $q_X\geq1/2$ (and
  $Q_X\leq2$). Note that for any infinite-dimensional Banach space
  $X$, $q_X\leq 1$. Indeed, for any $\delta>1$ there exists
  $f\colon\bn\to X$ with $1<\norm{f(m)-f(n)}<\delta$ for all $m<n$,
  and thus $X$ does not have the $Q(1,\delta)$-property.
\end{rem}

\begin{rem}
  Let $(e_i)$ and $(s_i)$ denote the unit vector basis and,
  respectively, the summing basis of $\co$. Thus,
  $s_n=\sum_{i=1}^ne_i$, $n\in\bn$. For $k\in\bn$ define
  $f\colon(\sizeeq{\bn}{k},d_I)\to\co$ by
  $f(\nb)=\sum_{i=1}^ks_{n_i}$. Note that
  $\norm{f(\mb)-f(\nb)}=d_I(\mb,\nb)=1$ for adjacent vertices, and
  hence $\lip(f)=1$. Moreover, $\norm{f(\mb)-f(\nb)}=k$ for all
  $\mb<\nb$ in $\sizeeq{\bn}{k}$. Thus $\co$ has the worst possible
  behaviour for property~$Q$. Indeed, this behaviour of the summing
  basis motivated the proof by Kalton that $\co$ does not uniformly or
  coarsely embed into a reflexive space. It was also the motivation
  behind the interlacing graph. Indeed,
  \[
  \norm{f(\mb)-f(\nb)}\leq d_I(\mb,\nb)\leq2\norm{f(\mb)-f(\nb)}
  \qquad\text{for all }\mb,\nb\in\sizeeq{\bn}{k}
  \]
  and for every $k\in\bn$.
\end{rem}

\begin{rem}
  If follows easily by scaling that for a Banach space $X$, the
  function $\Delta_X$ is linear, and thus $\Delta_X(\vare)=q_X\vare$
  for all $\vare>0$. Another observation of Kalton seems slightly less
  obvious, so we provide a proof next.
\end{rem}

\begin{prop}
  For a Banach space $X$, $\Delta_X(\vare)=\Delta_{B_X}(\vare)$ for
  all $0<\vare<1$.
\end{prop}
\begin{proof}
  Clearly, $\Delta_X(\vare)\leq\Delta_{B_X}(\vare)$ for any
  $\vare>0$. For the converse, we will first show that $\Delta_{B_X}$
  is ``linear'' on $(0,1)$. Fix $b$ with $0<b<1$. Given
  \[
  0<\vare_1<\vare_2\leq b\ ,
  \]
  write $\vare_2=a\vare_1$ for suitable $a>1$. We show that
  $\Delta_{B_X}(\vare_2)=a\Delta_{B_X}(\vare_1)$. This will show
  linearity on $(0,b]$, and hence on $(0,1)$. To show that
  $\Delta_{B_X}(\vare_2)=a\Delta_{B_X}(\vare_1)$, we may clearly
  assume that $a\leq b^{-1}$ as the general case will then follow by
  repeated applications of the special case. The inequality
  $\Delta_{B_X}(\vare_2)\geq a\Delta_{B_X}(\vare_1)$ follows by simple
  scaling. Conversely, assume that $B_X$ has the
  $Q(\vare_2,\delta)$-property. Let $k\in\bn$ and
  $f\colon\big(\sizeeq{\bn}{k},d_I\big)\to B_X$ be a function with
  $\lip(f)\leq\delta/a$. Then in particular,
  $\lip(f)\leq\delta$, and hence there exists $L\in\infin{\bn}$
  such that $\norm{f(\mb)-f(\nb)}<\vare_2$ for all $\mb<\nb$
  in $\sizeeq{L}{k}$. Enumerate $L$ as $l_1<l_2<\dots$ and set
  $\lb=\{l_1,\dots,l_k\}$. Define
  $g\colon\big(\sizeeq{\bn}{k},d_I\big)\to X$ by
  \[
  g(n_1,\dots,n_k)=\big(f(l_{k+n_1},\dots,l_{k+n_k})-f(\lb)\big)\cdot
  a\ .
  \]
  Since $a\vare_2\leq ab\leq 1$, it follows that $g$ takes values in
  $B_X$. Moreover, $\lip(g)\leq a\lip(f)\leq\delta$, and hence
  there exists $N\in\infin{\bn}$ such that
  $\norm{g(\mb)-g(\nb)}<\vare_2$ for all $\mb<\nb$ in
  $\sizeeq{N}{k}$. It follows that $\norm{f(\mb)-f(\nb)}<\vare_1$
  for all $\mb<\nb$ in $\sizeeq{P}{k}$, where
  $P=\{l_{k+n}:\,n\in N\}$. We have proved that $B_X$ has the
  $Q(\vare_1,\delta/a)$-property, and hence
  $a\cdot\Delta_{B_X}(\vare_1)\geq\Delta_{B_X}(\vare_2)$, as
  required.

  Having shown linearity of $\Delta_{B_X}$ on $(0,1)$, we now finish
  the proof by showing that $\Delta_{B_X}(\vare)\leq\Delta_X(\vare)$
  for any $0<\vare<1$.  Assume that $0\leq\delta<\Delta_{B_X}(\vare)$
  for some $\delta$ and $0<\vare<1$. Let $k\in\bn$ and
  $f\colon\big(\sizeeq{\bn}{k},d_I\big)\to X$ be a function with
  $\lip(f)\leq\delta$. Then for some $n\in\bn$, the function
  $n^{-1}f$ takes values in $B_X$, and moreover
  $\lip(n^{-1}f)\leq{}n^{-1}\delta<n^{-1}\Delta_{B_X}(\vare)=\Delta_{B_X}(n^{-1}\vare)$. Hence
  there exists $L\in\infin{\bn}$ such that
  $\norm{n^{-1}f(\mb)-n^{-1}f(\nb)}<n^{-1}\vare$ for all
  $\mb<\nb$ in $\sizeeq{L}{k}$. It follows that $X$ has the
  $Q(\vare,\delta)$-property, and thus
  $\Delta_X(\vare)\geq\Delta_{B_X}(\vare)$ as required.
\end{proof}

The following result shows that property~$Q$ is preserved under coarse
and uniform embeddings with Banach space domain and metric space
target. This is a trivial extension
of~\cite{kalton:07}*{Theorem~4.2}. We give the short proof for
convenience. Let us first recall that for a function
$f\colon(M,d_M)\to(N,d_N)$ between metric spaces, its \emph{modulus of
continuity $\omega_f$ }and \emph{compression modulus $\rho_f$ }are
defined, respectively, by
\[
\omega_f(t)=\sup \{ d_N(f(x),f(y)):\,d_M(x,y)\leq t\}
\]
and
\[
\rho_f(t)=\inf \{ d_N(f(x),f(y)):\,d_M(x,y)\geq t\}
\]
for $t\in[0,\infty)$. Note that $f$ is a uniform embedding if and only
if $\lim_{t\to0^+}\omega_f(t)=0$ and $\rho_f(t)>0$ for all $t>0$, and
that $f$ is a coarse embedding if and only if
$\lim_{t\to\infty}\rho_f(t)=\infty$ and $\omega_f(t)<\infty$ for all
$t>0$. Note also that if $M$ is a graph and $d_M$ is the graph
distance, then $\lip(f)=\omega_f(1)$.

\begin{prop}
  \label{prop:inheriting-Q}
  Let $h\colon (M,d_M)\to(N,d_N)$ be a function between metric spaces,
  and assume that $N$ has property~$Q$. Then for all $\vare>0$,
  \[
  \Delta_M(\vare)\geq \sup\{\delta\geq 0:\,
  q_N\cdot\rho_h(\vare)>\omega_h(\delta)\}\ .
  \]
  This has the following consequences.
  \begin{mylist}{(ii)}
  \item
    If $h$ is a uniform embedding, then $\Delta_M(\vare)>0$ for all
    $\vare>0$.
  \item
    If $h$ is a coarse embedding, then $\Delta_M(\vare)\to\infty$ as
    $\vare\to\infty$.
  \end{mylist}
  In particular, if a Banach space $X$ embeds coarsely into a metric
  space with property~$Q$ or if the unit ball $B_X$ embeds uniformly
  into a metric space with property~$Q$, then $X$ has property~$Q$.
\end{prop}

\begin{proof}
  Let $\delta\geq 0$ and $\vare>0$ be such that $M$ fails the
  $Q(\vare,\delta)$-property. Then there exist $k\in\bn$ and a
  function $f\colon\sizeeq{\bn}{k}\to M$ such that
  $\omega_f(1)\leq\delta$ and for all $L\in\infin{\bn}$ there exist
  $\mb<\nb$ in $\sizeeq{L}{k}$ with
  $d_M(f(\mb),f(\nb))\geq\vare$. Then
  $\omega_{hf}(1)\leq\omega_h(\delta)$ and for all $L\in\infin{\bn}$
  there exist $\mb<\nb$ in $\sizeeq{L}{k}$ with
  $d_N(hf(\mb),hf(\nb))\geq\rho_h(\vare)$. It follows that $N$ does
  not have the $Q(\rho_h(\vare),\omega_h(\delta))$-property, and hence
  \[
  \omega_h(\delta)\geq \Delta_N(\rho_h(\vare))\geq
  q_N\cdot\rho_h(\vare)\ .
  \]
  The result follows.
\end{proof}

Recall that for $k\in\bn$, the interlacing graph on the vertex set
$\sizeeq{\bn}{k}$ joins two vertices $\mb=\{m_1,\dots,m_k\}$ and
$\nb=\{n_1,\dots,n_k\}$ if and only if $\mb$ and $\nb$ are
interlacing, \ie either
\begin{mylist}{}
\item
  $m_1\leq n_1\leq m_2\leq n_2\leq\dots\leq m_k\leq n_k$ or
\item
  $n_1\leq m_1\leq n_2\leq m_2\leq\dots\leq n_k\leq m_k$.
\end{mylist}

\begin{rem}
  Using the above notation, let $m(x)=\abs{\{i:\,m_i\leq x\}}$,
  $x\in\bn$. The function $x\mapsto m(x)-n(x)$ measures the
  discrepancy between $\mb$ and $\nb$, which in turn measures the
  graph distance between $\mb$ and $\nb$:
  \[
  d_I(\mb,\nb)=\max_{x\in\bn}\big(m(x)-n(x)\big)-
  \min_{x\in\bn}\big(m(x)-n(x)\big)\ .
  \]
  It follows that the diameter of the interlacing graph on
  $\sizeeq{\bn}{k}$ is $k$, and moreover, $d_I(\mb,\nb)=k$ if and only
  if either $\mb$ lies strictly between two consecutive elements of
  $\nb$ including the degenerate cases $\mb<\nb$ (which means that
  $m_k<n_1$) and $\nb<\mb$, or vice versa. As noted earlier, the
  summing basis of $\co$ provides another description of the graph
  distance $d_I$.
\end{rem}

We will need the following stronger version of the interlacing
adjacency relation.

\begin{defn}
  Given $k\in\bn$ and $\mb,\nb\in\sizeeq{\bn}{k}$, we say that the
  pair $(\mb,\nb)$ is \emph{strongly interlacing }if either
  \begin{mylist}{}
  \item
    $m_1\leq n_1<m_2\leq n_2<\dots<m_k\leq n_k$ or
  \item
    $n_1\leq m_1<n_2\leq m_2<\dots< n_k\leq m_k$.
  \end{mylist}
\end{defn}

\begin{lem}
  \label{lem:strongly-interlacing}
  Given $k\in\bn$ and an interlacing pair $(\mb,\nb)$ in
  $\sizeeq{2\bn}{k}$, there exists $\pb\in\sizeeq{\bn}{k}$ such that
  the pairs $(\mb,\pb)$ and $(\nb,\pb)$ are strongly interlacing.
\end{lem}

\begin{proof}
  Without loss of generality, we may assume that
  \[
  m_1\leq n_1\leq m_2\leq n_2\leq\dots\leq m_k\leq n_k\ .
  \]
  Define $\pb=\{p_1,\dots,p_k\}$ in $\sizeeq{\bn}{k}$ by
  \[
  p_i=\begin{cases}%
  n_i-1 & \text{if }n_i=m_{i+1}\ ,\\
  n_i & \text{if }n_i<m_{i+1}\ .
  \end{cases}
  \]
  It is straightforward to verify that this works.
\end{proof}

We are now ready to state and prove the main result of this section.

\begin{thm}
  \label{thm:property-Q-implies-cut-stable}
  A metric space with property~$Q$ is cut stable.
\end{thm}

\begin{proof}  
  Let $(M,d_M)$ be a metric space with property~$Q$. We will show that
  $M$ is $C$-cut stable with $C=8Q_M+1$. Fix integers $1\leq l<k$,
  bounded functions $f\colon\sizeeq{\bn}{l}\to M$,
  $g\colon\sizeeq{\bn}{k-l}\to M$ and cuts $P,Q$ of $\{1,\dots,k\}$ of
  size $l$. Set
  \begin{align*}
    A &= \inf_{\nb\in\sizeeq{\bn}{k}}
    d_M\big(f(n_i:\,i\in P),g(n_i:\,i\in P^c)\big)\\
    B &=\sup_{\nb\in\sizeeq{\bn}{k}}
    d_M\big(f(n_i:\,i\in Q),g(n_i:\,i\in Q^c)\big)\ .
  \end{align*}
  Put $L=\{2nk:\,n\in\bn\}$ and let $L_f$ and $L_g$ be the Lipschitz
  constants of $f\restrict_{\sizeeq{L}{l}}$ and of
  $g\restrict_{\sizeeq{L}{k-l}}$, respectively. We first show that
  $L_f\leq4B$ and $L_g\leq4B$. To see this, set
  $L'=\{nk:\,n\in\bn\}$ and let $(\mb,\nb)$ be an interlacing
  pair in $\sizeeq{L}{l}$. By Lemma~\ref{lem:strongly-interlacing}
  there exists $\pb\in\sizeeq{L'}{l}$ such that the pairs
  $(\mb,\pb)$ and $(\pb,\nb)$ are strongly interlacing. We
  can then choose $m_1<\dots<m_k$ and $p_1<\dots<p_k$ in $L'$ such
  that $\mb=\{m_i:\,i\in Q\}$, $\pb=\{p_i:\,i\in Q\}$ and $m_i=p_i$
  for all $i\in Q^c$. It follows that
  \begin{multline*}
    d_M(f(\mb),f(\pb))\leq
    d_M\big(f(m_i:\,i\in Q),g(m_i:\,i\in Q^c)\big)\\
    +
    d_M\big(f(p_i:\,i\in Q),g(p_i:\,i\in Q^c)\big)
    \leq 2B\ .
  \end{multline*}
  Similarly, we get $d_M(f(\pb),f(\nb))\leq2B$, and hence
  $d_M(f(\mb),f(\nb))\leq4B$, as required. A similar argument shows that
  $L_g\leq4B$.

  Now fix $C>Q_M$. Since $M$ has property~$Q$, there exists
  $N\in\infin{L}$ such that $d_M(f(\mb),f(\nb))\leq CL_f$ for all
  $\mb,\nb\in\sizeeq{N}{l}$ and $d_M(g(\mb),g(\nb))\leq CL_g$ for all
  $\mb,\nb\in\sizeeq{N}{k-l}$. Now fix any $n_1<\dots<n_k$ in $N$. We
  get
  \begin{eqnarray*}
    A &\leq& d_M\big(f(n_i:\,i\in P),g(n_i:\,i\in P^c)\big)\\
    &\leq& d_M\big(f(n_i:i\in P),f(n_i:i\in Q)\big)\\
    && +
    d_M\big(f(n_i:i\in Q),g(n_i:i\in Q^c)\big)\\
    && +
    d_M\big(g(n_i:i\in Q^c),g(n_i:i\in P^c)\big)
    \leq (8C+1)B\ .
  \end{eqnarray*}
  The result follows.
\end{proof}

The implication that an upper stable metric space has property~$Q$ was
shown in~\cite{baudier:22}. Combining this with
Theorem~\ref{thm:property-Q-implies-cut-stable}, we obtain

\begin{cor}
  The following are equivalent for a metric space $(M,d_M)$.
  \begin{mylist}{(iii)}
  \item
    $M$ has property~$Q$.
  \item
    $M$ is cut stable.
  \item
    $M$ is upper stable.
  \end{mylist}
\end{cor}

Combining this with Proposition~\ref{prop:inheriting-Q}, we obtain the
following.

\begin{cor}
  If a Banach space $X$ embeds coarsely into an upper stable metric
  space, or if the unit ball $B_X$ embeds uniformly into an upper
  stable metric space, then $X$ is upper stable.
\end{cor}

\section{Reflexive spaces, upper stability and asymptotic structure}
\label{sec:upper-stability-and-asymptotic-structure}

Since reflexive spaces have property~$Q$ (Kalton~\cite{kalton:07}),
the following is an immediate consequence of
Theorem~\ref{thm:property-Q-implies-cut-stable}.

\begin{cor}
  \label{cor:reflexive-upper-stable}
  Every reflexive Banach space is upper stable.
\end{cor}

In this section, we give a direct proof of
Corollary~\ref{cor:reflexive-upper-stable} without using
property~$Q$. Instead, we shall relate stability and asymptotic
structure. We begin by recalling the notion of asymptotic structure
introduced by Maurey, Milman and
Tomczak-Jaegermann~\cite{mau-mil-tom:95}.

For each $k\in\bn$, let $\cE_k$ denote the set of norms on $\br^k$
with respect to which the standard basis $(e_i)_{i=1}^k$ is a
normalized, monotone basis. Note that $\cE_k$ is a compact subset of
$C\big([-1,1]^k\big)$ by the Arzel\`a--Ascoli theorem. Fix an
infinite-dimensional Banach space $X$ with
norm~$\norm{\cdot}$, and consider the following two-player game between
$\vS$ (subspace chooser) and $\vP$ (point chooser).  At each move of
the game $\vS$ first picks a finite-codimensional subspace $Y$ of $X$
and $\vP$ responds with a vector $x\in Y$ with $\norm{x}=1$. The game
finishes after exactly $k$ moves producing  a sequence
$Y_1,x_1,Y_2,x_2,\dots,Y_k,x_k$. Given a norm $N$ on $\br^k$ and
$\delta>0$, we say that \emph{$\vP$ wins the $(N,\delta)$-game }if
\begin{equation}
  \label{eq:within-delta-of-norm}
  \Bigabs{N\big(\tsum_{i=1}^ka_ie_i\big)-\bignorm{\tsum_{i=1}^ka_ix_i}}
  <\delta\qquad \text{for all }(a_i)_{i=1}^k\in[-1,1]^k\ .
\end{equation}
The \emph{$k^{\text{th}}$ asymptotic structure of $X$ }is the set
$\{X\}_k$ of all norms $N$ on $\br^k$ for which $\vP$ has a winning
strategy for the $(N,\delta)$-game for all $\delta>0$. One of the
results in~\cite{mau-mil-tom:95} states that $\{X\}_k$ is a non-empty
closed subset of $\cE_k$, and moreover it is the smallest closed
subset $\cF$ of $\cE_k$ such that for every $\delta>0$, $\vS$ has a
winning strategy for forcing $\cP$ to choose vectors $x_1,\dots,x_k$
such that~\eqref{eq:within-delta-of-norm} holds for some $N\in\cF$.

Winning strategies in the game described above naturally correspond to
trees. For $k\in\bn$ and $L\subset\bn$ we denote by $\sizeleq{L}{k}$
the set of all subsets of $L$ of size at most~$k$. By a
\emph{(countably branching) tree in $X$ of height $k$ }we mean a
function $t\colon\sizeleq{\bn}{k}\to X$ or
$t\colon\sizeleq{\bn}{k}\setminus\{\emptyset\}\to X$; in the former
case $t$ is a \emph{rooted }tree with root $t(\emptyset)$, and in the 
latter case $t$ is an \emph{unrooted }tree. We say $t$ is
\emph{normalized }if it takes values in the unit sphere $S_X$ of
$X$. Given a sequence $\cUb=(\cU_1,\dots,\cU_k)$ of free ultrafilters,
we say that $t$ is \emph{$\cUb$-weakly null }if for each
$1\leq i\leq k$ we have
\[
\V_{\cU_1}n_1\ \V_{\cU_2}n_2\ \dots\ \V_{\cU_{i-1}}n_{i-1}\quad
\wlim_{n,\cU_i}t(n_1,\dots,n_{i-1},n)=0\ .
\]
In the unrooted case, given
$\nb=\{n_1,\dots,n_k\}\in\sizeeq{\bn}{k}$ with $n_1<\dots<n_k$, the
sequence $\big(t(n_1,\dots,n_i)\big)_{i=1}^k$ is called a \emph{branch
}of $t$.

Since every weakly null sequence is eventually almost contained in any
finite-codimensional subspace, it follows that for every $k\in\bn$ and
$\delta>0$, every normalized, weakly null (with respect to some
sequence of free ultrafilters) unrooted tree in $X$ of height $k$ has
a branch $(x_1,\dots,x_k)$ satisfying~\eqref{eq:within-delta-of-norm}
for some $N\in\{X\}_k$. We shall need a somewhat stronger and more
precise formulation of this statement. Let us fix $k\in\bn$, a
sequence $\cUb=(\cU_1,\dots,\cU_k)$ of free ultrafilters and a
normalized, $\cUb$-weakly null tree $t$ in $X$. Define
\begin{eqnarray*}
  N_t^{\cUb} \colon X\oplus\br^k &\to&\br\\
  N_t^{\cUb}\Big(x\oplus\tsum_{i=1}^ka_ie_i\Big) &=& \lim_{\nb,\cUb}
  \Bignorm{x+\tsum_{i=1}^ka_it(n_1,\dots,n_i)}
\end{eqnarray*}
where $\lim_{\nb,\cUb}$ is an abbreviation for the iterated limit
$\lim_{n_1,\cU_1}\dots\lim_{n_k,\cU_k}$. It is clear that $N_t^{\cUb}$
is a seminorm on $X\oplus\br^k$.

\begin{lem}
  \label{lem:asymptotic-norm}
  In the situation described above, we have
  \[
  N_t^{\cUb}\Big(x\oplus\tsum_{i=1}^ja_ie_i\Big) \leq
  N_t^{\cUb}\Big(x\oplus\tsum_{i=1}^ka_ie_i\Big)
  \]
  for all $0\leq j\leq k$, $x\in X$ and $(a_i)_{i=1}^k\in\br^k$. It
  follows that
  \[
  N_t^{\cUb}\Big(0\oplus\tsum_{i=1}^ka_ie_i\Big) \leq
  2 N_t^{\cUb}\Big(x\oplus\tsum_{i=1}^ka_ie_i\Big)
  \]
  for all $x\in X$ and $(a_i)_{i=1}^k\in\br^k$.
\end{lem}

\begin{proof}
  The second inequality follows from the first one by
  triangle inequality. To see the first inequality, fix $x\in X$,
  $(a_i)_{i=1}^k\in\br^k$ and $0\leq j\leq k$. For every
  $\nb\in\sizeeq{\bn}{j}$, fix a norming functional
  $z^*(\nb)$ for $x+\sum_{i=1}^ja_it(n_1,\dots,n_i)$. It follows that
  \begin{align*}
    N_t^{\cUb}\Big(x\oplus\tsum_{i=1}^ja_ie_i\Big) &=
    \lim_{n_1,\cU_1}\dots\lim_{n_j,\cU_j}
    \Bignorm{x+\tsum_{i=1}^ja_it(n_1,\dots,n_i)} \\
    &= \lim_{n_1,\cU_1}\dots\lim_{n_j,\cU_j} z^*(n_1,\dots,n_j)
    \Big(x+\tsum_{i=1}^ja_it(n_1,\dots,n_i)\Big) \\
    &= \lim_{n_1,\cU_1}\dots\lim_{n_k,\cU_k} z^*(n_1,\dots,n_j)
    \Big(x+\tsum_{i=1}^ka_it(n_1,\dots,n_i)\Big) \\
    &\leq \lim_{n_1,\cU_1}\dots\lim_{n_k,\cU_k}
    \Bignorm{x+\tsum_{i=1}^ka_it(n_1,\dots,n_i)} \\
    &= N_t^{\cUb}\Big(x\oplus\tsum_{i=1}^ka_ie_i\Big)\ .
  \end{align*}
\end{proof}

\begin{rem}
  Lemma~\ref{lem:asymptotic-norm} shows that $N_t^{\cUb}$ is in fact a
  norm on $X\oplus\br^k$ extending the norm of $X$, and the unit
  vector basis of $\br^k$ is a normalized, monotone basis with respect
  to $N_t^{\cUb}$. By an earlier remark, the restriction of
  $N_t^{\cUb}$ to $\br^k$ is in the $k^{\text{th}}$ asymptotic
  structure of $X$. It is not very hard to show that if $X$ is
  separable and reflexive, then the converse also holds: every member
  of $\{X\}_k$ is of the form $N_t^{\cUb}$ for some normalized, weakly
  null tree $t$ in $X$ with $\cUb$ being arbitrary.
\end{rem}

Our next result approximates bounded functions, like those appearing
in Raynaud's characterization of stability, by weakly null trees.

\begin{lem}
  \label{lem:bdd-fn-as-tree}
  Assume $X$ is reflexive. Let $k\in\bn$, $\cUb=(\cU_1,\dots,\cU_k)$
  be a sequence of free ultrafilters and $f\colon\sizeeq{\bn}{k}\to X$
  a bounded function. Then there exist a normalized, $\cUb$-weakly null
  rooted tree $t$ in $X$ and non-negative scalars $a_0,a_1,\dots,a_k$
  such that
  \[
  \lim_{\nb,\cUb}\Bignorm{f(n_1,\dots,n_k)-\tsum_{i=0}^ka_it(n_1,\dots,n_i)}=0\ .
  \]
\end{lem}

\begin{proof}
  Define for each $1\leq j\leq k$ the function
  \begin{eqnarray*}
    \partial^j f\colon\sizeeq{\bn}{k-j} &\to& X\\
    \partial^j f(n_1,\dots,n_{k-j}) &=&
    \wlim_{n_{k-j+1},\cU_{k-j+1}}\dots\wlim_{n_k,\cU_k} f(n_1,\dots,n_k)\ .
  \end{eqnarray*}
  We will identify the function $\partial^kf$ on
  $\sizeeq{\bn}{0}=\{\emptyset\}$ with its image in $X$. Note that for
  each $\nb\in\sizeeq{\bn}{k}$ we have
  \begin{equation}
    \label{eq:bdd-fn-to-tree}
    f(n_1,\dots,n_k)=\partial^kf+\sum_{i=1}^k
    \big(\partial^{k-i}f(n_1,\dots,n_i)-\partial^{k-i+1}f(n_1,\dots,n_{i-1})\big)\ .
  \end{equation}
  Define an unrooted tree $a$ in $\br$ of height $k$ by setting
  \[
  a(n_1,\dots,n_i)=
  \bignorm{\partial^{k-i}f(n_1,\dots,n_i)-\partial^{k-i+1}f(n_1,\dots,n_{i-1})}
  \]
  for $1\leq i\leq k$ and $n_1<\dots<n_i$ in $\bn$. We next define
  $t(\emptyset)\in S_X$ and non-negative scalars $a_0,a_1,\dots,a_k$
  by
  \begin{align*}
    \partial^kf &= a_0t(\emptyset)\\
    a_i &= \lim_{n_1,\cU_1}\dots\lim_{n_i,\cU_i} a(n_1,\dots,n_i)
    \qquad\text{for }1\leq i\leq k\ .
  \end{align*}
  We now define our rooted tree $t$ in $S_X$ as follows. We have
  already defined $t(\emptyset)$. Fix $1\leq i\leq k$. If $a_i=0$,
  then for any $n_1<\dots<n_{i-1}$ in $\bn$, we choose
  $\big(t(n_1,\dots,n_{i-1},n)\big)_{n=n_{i-1}+1}^\infty$ to be an
  arbitrary normalized weakly null sequence. If $a_i\neq0$, then for
  $n_1<\dots<n_i$ we set
  \[
  t(n_1,\dots,n_i)=%
  \begin{cases}
    \frac{\partial^{k-i}f(n_1,\dots,n_i)-\partial^{k-i+1}f(n_1,\dots,n_{i-1})}{a(n_1,\dots,n_i)}
    & \text{if }a(n_1,\dots,n_i)\neq0\ ,\\
    \text{an arbitrary norm-$1$ vector} & \text{otherwise.}
  \end{cases}
  \]
  If $a_i\neq0$, then $\V_{\cU_1}n_1\dots\V_{\cU_i}n_i$
  $a(n_1,\dots,n_i)>a_i/2$, from which it follows that $t$ is
  $\cUb$-weakly null. Now observe that for $n_1<\dots<n_k$ we have
  \begin{multline*}
    f(n_1,\dots,n_k)-\sum_{i=0}^ka_it(n_1,\dots,n_i)\\
    =\sum_{i=1}^k \Big[
      \big(\partial^{k-i}f(n_1,\dots,n_i)-\partial^{k-i+1}f(n_1,\dots,n_{i-1})\big)
      - a(n_1,\dots,n_i)t(n_1,\dots,n_i) \Big]\\
    +\sum_{i=1}^k\big(a(n_1,\dots,n_i)-a_i\big)t(n_1,\dots,n_i)\ ,
  \end{multline*}
  where we used~\eqref{eq:bdd-fn-to-tree}. It follows immediately that
  \[
  \lim_{\nb,\cUb}\Bignorm{f(n_1,\dots,n_k)-\tsum_{i=0}^ka_it(n_1,\dots,n_i)}=0\ ,
  \]
  as required.
\end{proof}

An important tool in showing that reflexive spaces are cut stable will
be the following operation on trees.

\begin{defn}
  Let $k,l,m\in\bn$ with $k=l+m$. Let $s$ and $t$ be trees in $X$
  of height $l$ and $m$, respectively. For a cut $P$ of
  $\{1,\dots,k\}$ of size $l$, the \emph{$P$-intertwining of $s$ and
  $t$ }is the unrooted tree $s\itws_P t$ in $X$ of height $k$ defined by
  \[
  (s\itws_P t)(n_1,\dots,n_j)=%
  \begin{cases}
    s(n_i:\,i\in P,\ i\leq j) & \text{if }j\in P\ ,\\
    t(n_i:\,i\in P^c,\ i\leq j) & \text{if }j\in P^c
  \end{cases}
  \]
  for $1\leq j\leq k$ and $n_1<\dots<n_j$ in $\bn$.
\end{defn}

Note that if $P=\{p_1,\dots,p_l\}$ with $p_1<\dots<p_l$ and
$P^c=\{q_1,\dots,q_m\}$ with $q_1<\dots<q_m$, then
\begin{align*}
  (s\itws_P t)(n_1,n_2,\dots,n_{p_i}) &= s(n_{p_1},\dots,n_{p_i}) &&
  \text{for }1\leq i\leq l\ ,\\
  (s\itws_P t)(n_1,n_2,\dots,n_{q_j}) &= t(n_{q_1},\dots,n_{q_j}) &&
  \text{for }1\leq j\leq m\ .
\end{align*}
Given a sequence $\cUb=(\cU_1,\dots,\cU_k)$ of free ultrafilters, let
$\cUb_P=(\cU_{p_1},\dots,\cU_{p_l})$ and
$\cUb_{P^c}=(\cU_{q_1},\dots,\cU_{q_m})$. Note that if $s$ is
$\cUb_P$-weakly null and $t$ is $\cUb_{P^c}$-weakly null, then
$s\itws_Pt$ is $\cUb$-weakly null.

Although the standard basis of $\br^k$ is not in general unconditional
with respect to a norm in the $k^{\text{th}}$ asymptotic structure of
a Banach space, we have the following result.

\begin{prop}
  \label{prop:intertwining-and-unconditionality}
  Assume $X$ is reflexive. Let $k,l,m\in\bn$ with $k=l+m$ and $P$ be a
  cut of $\{1,\dots,k\}$ of size $l$. Let $\cUb=(\cU_1,\dots,\cU_k)$
  be a sequence of free ultrafilters. Let $s$ and $t$ be normalized
  trees in $X$ of height $l$ and $m$, respectively, such that
  $s$ is $\cUb_P$-weakly null and $t$ is $\cUb_{P^c}$-weakly
  null. Then
  \begin{equation}
    \label{eq:project-P}
    N^{\cUb}_{s\itws_P t}\Big(\tsum_{i\in P}a_ie_i\Big) \leq
    N^{\cUb}_{s\itws_P t}\Big(\tsum_{i=1}^ka_ie_i\Big)
  \end{equation}
  and
  \begin{equation}
    \label{eq:project-PC}
    N^{\cUb}_{s\itws_P t}\Big(\tsum_{i\in P^c}a_ie_i\Big) \leq
    N^{\cUb}_{s\itws_P t}\Big(\tsum_{i=1}^ka_ie_i\Big) 
  \end{equation}
  for all $(a_i)_{i=1}^k\in\br^k$.
\end{prop}

\begin{proof}
  Write $P$ as $P=\{p_1,\dots,p_l\}$ with $p_1<\dots<p_l$ and $P^c$ as
  $P^c=\{q_1,\dots,q_m\}$ with $q_1<\dots<q_m$. Put $u=s\itws_Pt$. We
  claim that
  \begin{equation}
    \label{eq:one-step-projection}
    N^{\cUb}_u \Big(\tsum_{i=1}^la_{p_i}e_{p_i}
    +\tsum_{j=1}^{m-1}a_{q_j}e_{q_j}\Big) \leq N^{\cUb}_u
    \Big(\tsum_{i=1}^ka_ie_i\Big)\ .
  \end{equation}
  for all $(a_i)_{i=1}^k\in\br^k$. Using the proof of this claim, we
  can repeatedly eliminate coefficients to
  deduce~\eqref{eq:project-P}. The proof of~\eqref{eq:project-PC} is
  similar.

  To see the claim, fix $\vare>0$, $n_1<\dots<n_{q_m-1}$ in $\bn$ and
  $L\in\cU_{q_m}$ such that
  \[
  0\in\overline{\{u(n_1,n_2,\dots,n_{q_m}):\,n_{q_m}\in L\}}^w
  \]
  and for all $n_{q_m}\in L$ and for all $(a_i)_{i=1}^k\in[-1,1]^k$,
  we have
  \begin{multline*}
    \biggabs{N^{\cUb}_u\Big(\tsum_{i=1}^ka_ie_i\Big) -\\
      \lim_{n_{q_m+1},\cU_{q_m+1}}\dots\lim_{n_k,\cU_k}
      \Bignorm{\sum_{i=1}^ka_i u(n_1,\dots,n_i)}} <\vare\ .
  \end{multline*}
  Note that $\V_{\cU_1}n_1\dots\V_{\cU_{q_m-1}}n_{q_m-1}$ such
  $L\in\cU_{q_m}$ exists. By Mazur's theorem there is a finite set
  $F\subset L$ and a convex combination
  \[
  \sum_{n\in F}\tau_n u(n_1,n_2,\dots,n_{q_m-1},n)
  \]
  with norm
  \[
  \Bignorm{\tsum_{n\in
      F}\tau_n u(n_1,n_2,\dots,n_{q_m-1},n)}<\vare\ .
  \]
  Since $F$ is finite, there exist $\max F<n_{q_m+1}<\dots<n_k$ such
  that
  \[
  \biggabs{N^{\cUb}_u\Big(\tsum_{i=1}^ka_ie_i\Big) -\\
    \Bignorm{\sum_{i=1}^ka_i u(n_1,\dots,n_i)}} <2\vare
  \]
  for all $n_{q_m}\in F$ and for all $(a_i)_{i=1}^k\in[-1,1]^k$. Now
  fix $(a_i)_{i=1}^k\in[-1,1]^k$ and observe that
  \begin{align*}
    N^{\cUb}_u \Big(\tsum_{i=1}^la_{p_i}e_{p_i}
    &+\tsum_{j=1}^{m-1}a_{q_j}e_{q_j}\Big)\\
    &\leq
    \Bignorm{\tsum_{i=1}^{q_m-1}a_iu(n_1,\dots,n_i)+\tsum_{i=q_m+1}^ka_iu(n_1,\dots,n_i)} 
    +2\vare\\
    &\leq \Bignorm{\tsum_{n_{q_m}\in
        F}\tau_{n_{q_m}}\tsum_{i=1}^ka_iu(n_1,\dots,n_i)}+3\vare\\
    &\leq N^{\cUb}_u \Big(\tsum_{i=1}^ka_ie_i\Big)+5\vare\ .
  \end{align*}
  Here it is crucial that for $i>q_m$, $u(n_1,\dots,n_i)$ does not
  depend on $n_{q_m}$.
  This completes the proof of the claim.
\end{proof}

We are now ready to prove the main result of this section.

\begin{thm}
  \label{thm:reflexive-cut-stable}
  Every reflexive Banach space is $5$-cut stable, and thus, in
  particular, upper stable.
\end{thm}

\begin{proof}
  Fix $k,l,m\in\bn$ with $k=l+m$, a sequence
  $\cUb=(\cU_1,\dots,\cU_k)$ of free ultrafilters and cuts $P,Q$ of
  $\{1,\dots,k\}$ each of size $l$. Let $X$ be an
  infinite-dimensional, reflexive Banach space and
  $f\colon\sizeeq{\bn}{l}\to X$ and $g\colon\sizeeq{\bn}{m}\to X$ be
  bounded functions. We will show that
  \begin{multline}
    \label{eq:cut-stable-claim}
    \lim_{\nb,\cUb}\norm{f(n_i:\,i\in P)+g(n_i:\,i\in P^c)}\\
    \leq 5 \lim_{\nb,\cUb}\norm{f(n_i:\,i\in Q)+g(n_i:\,i\in Q^c)}
  \end{multline}
  in the special case when the ultrafilters $\cU_i$ are all the
  same. This will then complete the proof.

  Write $P=\{p_1\,\dots,p_l\}$ with $p_1<\dots<p_l$ and
  $P^c=\{q_1,\dots,q_m\}$ with $q_1<\dots<q_m$. By
  Lemma~\ref{lem:bdd-fn-as-tree} we may assume that there exist
  normalized, rooted trees $s$ and $t$ in $X$ of height $l$ and $m$,
  respectively, such that $s$ is $\cUb_P$-weakly null, $t$ is
  $\cUb_{P^c}$-weakly null, and there exist non-negative scalars
  $a_0,a_1,\dots,a_l$ and $b_0,b_1,\dots,b_m$ such that
  \[
  f(n_1,\dots,n_l)=a_0s(\emptyset)+\sum_{i=1}^la_is(n_1,\dots,n_i)
  \qquad\text{for }n_1<\dots<n_l
  \]
  and
  \[
  g(n_1,\dots,n_m)=b_0t(\emptyset)+\sum_{j=1}^mb_jt(n_1,\dots,n_j)
  \qquad\text{for }n_1<\dots<n_m\ .
  \]
  Let $u=s\itws_Pt$. Set $c_{p_i}=a_i$ for $1\leq i\leq l$ and
  $c_{q_j}=b_j$ for $1\leq j\leq m$. We shall frequently use the fact
  that
  \[
  N^{\cUb}_u\Big(\tsum_{i\in
    P}c_ie_i\Big)=N^{\cUb_P}_s\Big(\tsum_{i=1}^la_ie_i\Big)\ .
  \]
  Now observe that for $\nb\in\sizeeq{\bn}{k}$ we have
  \[ 
  f(n_i:\,i\in P)+g(n_i:\,i\in
  P^c)=a_0s(\emptyset)+b_0t(\emptyset)+\sum_{i=1}^k c_i
  u(n_1,\dots,n_i)\ .
  \]
  It follows that
  \begin{equation}
    \label{eq:sum-as-intertwined-limit}
    \lim_{\nb,\cUb}\norm{f(n_i:\,i\in P)+g(n_i:\,i\in P^c)}=N^{\cUb}_u
    \Big( \big(a_0s(\emptyset)+b_0t(\emptyset)\big)\oplus \tsum_{i=1}^k
    c_ie_i\Big)\ .
  \end{equation}
  By the triangle inequality, we have
  \begin{multline}
    \label{eq:upper-bound}
    N^{\cUb}_u \Big( \big(a_0s(\emptyset)+b_0t(\emptyset)\big)\oplus
    \tsum_{i=1}^k c_ie_i\Big)\\
    \leq \norm{a_0s(\emptyset)+b_0t(\emptyset)} +
    N^{\cUb_P}_s\Big(\tsum_{i=1}^la_ie_i\Big) + N^{\cUb_{P^c}}_t\Big(\tsum_{j=1}^mb_je_j\Big) 
  \end{multline}
  On the other hand, it follows from Lemma~\ref{lem:asymptotic-norm}
  and Proposition~\ref{prop:intertwining-and-unconditionality} that
  \begin{multline}
    \label{eq:lower-bound}
    \norm{a_0s(\emptyset)+b_0t(\emptyset)} +
    N^{\cUb_P}_s\Big(\tsum_{i=1}^la_ie_i\Big) +
    N^{\cUb_{P^c}}_t\Big(\tsum_{j=1}^mb_je_j\Big) \\
    \leq
    5\cdot N^{\cUb}_u
    \Big(\big(a_0s(\emptyset)+b_0t(\emptyset)\big)\oplus
    \tsum_{i=1}^k c_ie_i\Big)\ .
  \end{multline}
  When the ultrafilters $\cU_i$ are all identical, then
  \[
  N^{\cUb_P}_s\Big(\tsum_{i=1}^la_ie_i\Big) =
  N^{\cUb_Q}_s\Big(\tsum_{i=1}^la_ie_i\Big)
  \]
  and
  \[
  N^{\cUb_{P^c}}_t\Big(\tsum_{j=1}^mb_je_j\Big) =
  N^{\cUb_{Q^c}}_t\Big(\tsum_{j=1}^mb_je_j\Big)\ .
  \]
  Hence~\eqref{eq:cut-stable-claim} follows immediately
  from~\eqref{eq:sum-as-intertwined-limit},~\eqref{eq:upper-bound}
  and~\eqref{eq:lower-bound} applied to $u=s\itws_Pt$ or $u=s\itws_Qt$
  as appropriate.
\end{proof}


\begin{bibdiv}
\begin{biblist}


\bib{aldous:81}{article}{
  author={Aldous, D. J.},
  title={Subspaces of $L\sp {1}$, via random measures},
  journal={Trans. Amer. Math. Soc.},
  volume={267},
  date={1981},
  number={2},
  pages={445--463},
}

\bib{baudier:22}{article}{
  author={Baudier, F.},
  title={Barycentric gluing and geometry of stable metrics},
  journal={Rev. R. Acad. Cienc. Exactas F\'{\i }s. Nat. Ser. A Mat. RACSAM},
  volume={116},
  date={2022},
  number={1},
  pages={Paper No. 37, 48},
}

\bib{bcdkrsz:17}{article}{
  author={Baudier, F.},
  author={Causey, R.},
  author={Dilworth, S.},
  author={Kutzarova, D.},
  author={Randrianarivony, N. L.},
  author={Schlumprecht, Th.},
  author={Zhang, S.},
  title={On the geometry of the countably branching diamond graphs},
  journal={J. Funct. Anal.},
  volume={273},
  date={2017},
  number={10},
  pages={3150--3199},
}

\bib{bau-gar:21}{article}{
  author={Baudier, F. P},
  author={Gartland, C.},
  title={Umbel convexity and the geometry of trees},
  year={2021},
  eprint={arXiv:2103.16011 [math.MG]},
}

\bib{bau-kal-lan:10}{article}{
  author={Baudier, F.},
  author={Kalton, N. J.},
  author={Lancien, G.},
  title={A new metric invariant for Banach spaces},
  journal={Studia Math.},
  volume={199},
  date={2010},
  number={1},
  pages={73--94},
}

\bib{bau-lan:15}{article}{
  author={Baudier, F.},
  author={Lancien, G.},
  title={Tight embeddability of proper and stable metric spaces},
  journal={Anal. Geom. Metr. Spaces},
  volume={3},
  date={2015},
  number={1},
  pages={140--156},
}

\bib{bau-lan-mot-sch:21a}{article}{
  author={Baudier, F.},
  author={Lancien, G.},
  author={Motakis, P.},
  author={Schlumprecht, Th.},
  title={A new coarsely rigid class of Banach spaces},
  journal={J. Inst. Math. Jussieu},
  volume={20},
  date={2021},
  number={5},
  pages={1729--1747},
}

\bib{bau-lan-mot-sch:21b}{article}{
  author={Baudier, F.},
  author={Lancien, G.},
  author={Motakis, P.},
  author={Schlumprecht, Th.},
  title={The geometry of Hamming-type metrics and their embeddings into Banach spaces},
  journal={Israel J. Math.},
  volume={244},
  date={2021},
  number={2},
  pages={681--725},
}

\bib{bau-lan-sch:18}{article}{
  author={Baudier, F.},
  author={Lancien, G.},
  author={Schlumprecht, Th.},
  title={The coarse geometry of Tsirelson's space and applications},
  journal={J. Amer. Math. Soc.},
  volume={31},
  date={2018},
  number={3},
  pages={699--717},
}

\bib{ben-lin:00}{book}{
  author={Benyamini, Y.},
  author={Lindenstrauss, J.},
  title={Geometric nonlinear functional analysis. Vol. 1},
  series={American Mathematical Society Colloquium Publications},
  volume={48},
  publisher={American Mathematical Society, Providence, RI},
  date={2000},
  pages={xii+488},
  isbn={0-8218-0835-4},
}

\bib{bra-lan-pet-pro:23}{article}{
  author={Braga, B. M.},
  author={Lancien, G.},
  author={Petitjean, C.},
  author={Proch\'{a}zka, A.},
  title={On Kalton's interlaced graphs and nonlinear embeddings into dual Banach spaces},
  journal={J. Topol. Anal.},
  volume={15},
  date={2023},
  number={2},
  pages={467--494},
}


\bib{dav-fig-joh-pel:74}{article}{
  author={Davis, W. J.},
  author={Figiel, T.},
  author={Johnson, W. B.},
  author={Pe\l czy\'{n}ski, A.},
  title={Factoring weakly compact operators},
  journal={J. Functional Analysis},
  volume={17},
  date={1974},
  pages={311--327},
}

\bib{fovelle:23}{article}{
  author={Fovelle, A.},
  title={Hamming graphs and concentration properties in non-quasi-reflexive Banach spaces},
  year={2023},
  eprint={arXiv:2106.04297 [math.FA]},
}

\bib{garling:82}{article}{
  author={Garling, D. J. H.},
  title={Stable Banach spaces, random measures and Orlicz function spaces},
  conference={ title={Probability measures on groups}, address={Oberwolfach}, date={1981}, },
  book={ series={Lecture Notes in Math.}, volume={928}, publisher={Springer, Berlin-New York}, },
  isbn={3-540-11501-3},
  date={1982},
  pages={121--175},
}

\bib{gromov:93}{article}{
  author={Gromov, M.},
  title={Asymptotic invariants of infinite groups},
  conference={ title={Geometric group theory, Vol. 2}, address={Sussex}, date={1991}, },
  book={ series={London Math. Soc. Lecture Note Ser.}, volume={182}, publisher={Cambridge Univ. Press, Cambridge}, },
  isbn={0-521-44680-5},
  date={1993},
  pages={1--295},
}

\bib{hei-man:82}{article}{
  author={Heinrich, S.},
  author={Mankiewicz, P.},
  title={Applications of ultrapowers to the uniform and Lipschitz classification of Banach spaces},
  journal={Studia Math.},
  volume={73},
  date={1982},
  number={3},
  pages={225--251},
}

\bib{kalton:07}{article}{
  author={Kalton, N. J.},
  title={Coarse and uniform embeddings into reflexive spaces},
  journal={Q. J. Math.},
  volume={58},
  date={2007},
  number={3},
  pages={393--414},
}

\bib{kalton:11}{article}{
  author={Kalton, N. J.},
  title={Lipschitz and uniform embeddings into $\ell _\infty $},
  journal={Fund. Math.},
  volume={212},
  date={2011},
  number={1},
  pages={53--69},
}

\bib{kalton:12}{article}{
  author={Kalton, N. J.},
  title={The uniform structure of Banach spaces},
  journal={Math. Ann.},
  volume={354},
  date={2012},
  number={4},
  pages={1247--1288},
}

\bib{kalton:13}{article}{
  author={Kalton, N. J.},
  title={Uniform homeomorphisms of Banach spaces and asymptotic structure},
  journal={Trans. Amer. Math. Soc.},
  volume={365},
  date={2013},
  number={2},
  pages={1051--1079},
}

\bib{kal-ran:08}{article}{
  author={Kalton, N. J.},
  author={Randrianarivony, N. L.},
  title={The coarse Lipschitz geometry of $l_p\oplus l_q$},
  journal={Math. Ann.},
  volume={341},
  date={2008},
  number={1},
  pages={223--237},
}

\bib{kri-mau:81}{article}{
  author={Krivine, J.-L.},
  author={Maurey, B.},
  title={Espaces de Banach stables},
  language={French, with English summary},
  journal={Israel J. Math.},
  volume={39},
  date={1981},
  number={4},
  pages={273--295},
}

\bib{lan-pet-pro:20}{article}{
  author={Lancien, G.},
  author={Petitjean, C.},
  author={Proch\'{a}zka, A.},
  title={On the coarse geometry of James spaces},
  journal={Canad. Math. Bull.},
  volume={63},
  date={2020},
  number={1},
  pages={77--93},
}

\bib{lan-raj:18}{article}{
  author={Lancien, G.},
  author={Raja, M.},
  title={Asymptotic and coarse Lipshitz structures of quasi-reflexive Banach spaces},
  journal={Houston J. Math.},
  volume={44},
  date={2018},
  number={3},
  pages={927--940},
}

\bib{mau-mil-tom:95}{article}{
  author={Maurey, B.},
  author={Milman, V. D.},
  author={Tomczak-Jaegermann, N.},
  title={Asymptotic infinite-dimensional theory of Banach spaces},
  conference={ title={Geometric aspects of functional analysis}, address={Israel}, date={1992--1994}, },
  book={ series={Oper. Theory Adv. Appl.}, volume={77}, publisher={Birkh\"auser}, place={Basel}, },
  date={1995},
  pages={149--175},
}

\bib{netillard:18}{article}{
  author={Netillard, F.},
  title={Coarse Lipschitz embeddings of James spaces},
  journal={Bull. Belg. Math. Soc. Simon Stevin},
  volume={25},
  date={2018},
  number={1},
  pages={71--84},
}

\bib{now-yu:12}{book}{
  author={Nowak, P. W.},
  author={Yu, G.},
  title={Large scale geometry},
  series={EMS Textbooks in Mathematics},
  publisher={European Mathematical Society (EMS), Z\"{u}rich},
  date={2012},
  pages={xiv+189},
  isbn={978-3-03719-112-5},
}

\bib{raynaud:83}{article}{
  author={Raynaud, Y.},
  title={Espaces de Banach superstables, distances stables et hom\'{e}omorphismes uniformes},
  language={French, with English summary},
  journal={Israel J. Math.},
  volume={44},
  date={1983},
  number={1},
  pages={33--52},
}



\end{biblist}
\end{bibdiv}
\end{document}